\documentclass[11pt]{article}
\usepackage{indentfirst}
\usepackage[dvips]{graphicx} 
\usepackage{latexsym}
\usepackage{amsfonts}
\usepackage{amssymb}
\usepackage{amsmath}
\usepackage{amsthm}

\usepackage{geometry} 
\numberwithin{equation}{section}

\usepackage{todonotes}

\usepackage{xcolor,pict2e}
\usepackage{graphicx}
\usepackage{subfig}
\usepackage{float}
\usepackage{caption}
\usepackage{subfig}

\theoremstyle{plain}
\newtheorem{thm}{Theorem}[section]
\newtheorem{alg}{Algorithm}[section]

\newtheorem{lem}[thm]{Lemma}

\theoremstyle{definition}

\newtheorem{rem}{Remark}[section]

\title{Application of gradient descent algorithms based on geodesic distances}
\author{
Xiaomin Duan\\
{\small \it School of Science, Dalian Jiaotong University, Dalian 116028, P.R. China}\\
            {\small \it  E-mail: dxmhope@djtu.edu.cn}
 \\ Huafei Sun\thanks{Corresponding author}  \\
{\small \it School of Mathematics and Statistics, Beijing Institute
              of Technology, Beijing 100081, P.R. China}\\
            {\small \it  E-mail: huafeisun@bit.edu.cn}
\\ Linyu Peng\\
{\small \it Waseda Institute for Advanced Study, Waseda University, Tokyo 169-8050, Japan}\\
          {\small \it  E-mail:  l.peng@aoni.waseda.jp}
            }
\date{}

\begin{document}

\maketitle

{\bf Abstract.} In this paper, the Riemannian gradient algorithm and the natural
gradient algorithm  are
applied to solve descent direction problems on  the manifold of
positive definite Hermitian matrices, where the geodesic distance is considered as the cost function. The first proposed problem is control for positive definite
Hermitian matrix systems whose outputs only depend on their inputs.
The geodesic distance is adopted as the difference of  the output matrix and
 the target matrix. The controller to adjust the input is obtained such that the output
  matrix is as close as possible to the target matrix.
 We show the trajectory of the control input on the manifold using the Riemannian gradient algorithm. The second application is to compute the Karcher mean of a finite set of given Toeplitz  positive definite Hermitian matrices,
  which is defined as the minimizer of the sum of geodesic distances. To obtain more
  efficient iterative algorithm compared with traditional ones, a natural gradient algorithm is proposed to compute the Karcher mean. Illustrative simulations are provided to show the computational behavior of the  proposed algorithms.


 \vspace{0.2cm} {\bf Keywords:}  Riemannian  gradient algorithm, natural gradient algorithm, system control,
 Karcher mean, Toeplitz positive definite Hermitian matrix

\vspace{0.2cm} {\bf MSC:} 26E60, 53B20,
93A30, 22E60

\newpage
\section{Introduction}

Gradient adaptation is commonly applied to minimize a  cost
function by adjusting the  parameters. Although it is often easy to
implement, convergence speed of  the gradient adaptation can be
slow when the slope of the cost function varies widely for a small
change of the parameters. To overcome the weakness of slow
convergence, Amari et al. (\cite{amari12,amari11}) proposed
the natural gradient algorithm which defines the steepest descent
direction in Riemannian spaces based on the Riemannian structure of the parameter spaces.
Amari also proved that the  natural gradient is asymptotically
Fisher-efficient for the maximum likelihood estimation, implying that it has almost the same performance as the optimal
batch estimation of  the  parameters. The natural gradient algorithm has been widely applied into, for instance neural network, optimal control, offering a new way to solve such problems more effectively, cf. \cite{Li,Bastian,zhang,Zhao}.

Although  the natural gradient algorithm defines the steepest descent
direction,
iteration trajectory of  the
parameters  is not necessary the shortest, not to mention the difficulty to computer inverse of the metric.
 These problems are solved by using the Riemannian gradient algorithm in particular for matrix manifolds, separately introduced by Barbaresco \cite{Barba2} and Lenglet et al. \cite{Lenglet} with wide applications, e.g. \cite{Duan2014,moak1}. It is realized that the iterative path of each parameter is along
its geodesic, though the descent speed of the algorithm is not the
fastest in some cases and the scope of  the application is sometimes
limited.

In this paper, the set of $n\times n$
positive definite Hermitian matrices is defined  as a manifold
$P(n)$, whose geodesic connecting two matrices was studied in Moakher \cite{moakher}. Noting that the  geodesic distance
represents the  infimum about length functions of the curves connecting two matrices, we apply both the
 Riemannian gradient algorithm and  the natural gradient algorithm to  solve the  descent direction problems
taking  the geodesic distance as a cost function.
The first problem is control of  positive definite Hermitian matrix systems
  on manifold $P(n)$ using different gradient algorithms.
   Supposing the output is only determined by the control
  input, we  take the geodesic distance as the measure of the output matrix and  the target matrix. Controller to adjust the control input is shown, such that the  output
 matrix is as close as possible to the target matrix. Trajectory of the control inout is also obtained.
  Second, both gradient algorithms are used to computer  the Karcher mean of a finite set of given Toeplitz positive definite Hermitian matrices when sum of geodesic distances between any two matrices is viewed as the cost function. The examples show that convergence rate of the natural gradient algorithm is faster than that of the Riemannian gradient algorithm.

\section{Riemannian metric and geodesics on manifold $P(n)$}



Let $M(n,\mathbb{C})$ be the set of $n\times n$ complex matrices and
$GL(n,\mathbb{C})$ be its subset containing only non-singular
matrices. It is well known that $GL(n,\mathbb{C})$ is a
Lie group, roughly speaking a group on which a differentiable manifold can also be defined.  Its Lie algebra is denoted by $gl(n,\mathbb{C})$.  In
$M(n,\mathbb{C})$, one has the Euclidean inner product, known as the
Frobenius inner product  defined by
\begin{equation}
\langle A,B\rangle=\operatorname{tr}\left(A^{\operatorname{H}}B\right),
\end{equation}
where $\operatorname{tr}$ stands for the trace and the superscript
$A^{\operatorname{H}}$ denotes the conjugate transport of matrix $A$. The associated norm is defined as
\begin{equation}
\|A\|=\langle A,A\rangle^{\frac{1}{2}}.
\end{equation}
With the above defined inner product,
$M(n,\mathbb{C})$ is flat.

It is well known that the set $P(n)$ of all $n\times n$
positive definite Hermitian matrices is an $n^2$-dimensional manifold.
Let us denote the space of all $n\times n$ Hermitian matrices by $H(n)$.
The exponential map from $H(n)$ to $P(n)$
is one-to-one and onto. As $P(n)$ is an open subset of
$H(n)$, for each $A\in P(n)$ we identify the set $T_AP(n)$
of tangent vectors to $P(n)$  at $A$. Moreover,
the Riemannian  metric on $P(n)$ is given by
\begin{equation}\label{metric}
g_{A}(X,Y):=\langle X,Y\rangle_A=\langle A^{-1}X,A^{-1}Y\rangle_I=\operatorname{tr}\left(A^{-1}XA^{-1}Y\right),
\end{equation}
where $I$ denotes the identity element of $P(n)$ and $X, Y\in T_AP(n)$.
The positive definiteness of this metric is a consequence of the positive
definiteness of the Frobenius inner product.

Let $[0,1]$ be a closed interval in $\mathbb{R}$, and
$\gamma:[0,1]\rightarrow P(n)$ be a sufficiently smooth
curve on manifold $P(n)$. The length of $\gamma(t)$ is
\begin{equation}
 \ell(\gamma(t)):=\int_0^1
 \sqrt{\langle\dot{\gamma}(t),\dot{\gamma}(t)\rangle_{\gamma(t)}}\operatorname{d}\!t=\int_0^1
\sqrt{\operatorname{tr}\left({\gamma}^{-1}(t)\dot{\gamma}(t)\right)^2}\operatorname{d}\!t.
\end{equation}
 The geodesic distance between two matrices $A$ and $B$ on manifold
$P(n)$ is the minimal length
of curves connecting them:
\begin{equation}\label{inf}
d(A,B):=\inf \left\{\ell(\gamma)\mid\gamma:[0,1]\rightarrow
P(n)\ \text{ with } \gamma(0)=A,\gamma(1)=B\right\}.
\end{equation}

It transpires that length-minimizing smooth curves are geodesics,
thus the infimum of (\ref{inf}) is achieved by geodesic curves. The
Hopf--Rinow theorem \cite{hopf} implies that $P(n)$ is geodesically
complete. This means that the interval $[0, 1]$ can be extended to
$(-\infty,+\infty)$ and hence, for any given pair $A,B$, we can find
a geodesic curve ${\gamma}(t)$ such that ${\gamma}(0)= A$  and
${\gamma}(1) = B$, namely by taking the initial velocity as
$\dot{\gamma}(0)=
A^{\frac{1}{2}}\ln\left(A^{-\frac{1}{2}}BA^{-\frac{1}{2}}\right)A^{\frac{1}{2}}$.
Note that the length $ \ell(\gamma(t))$ is invariant under congruent
transformation  $\gamma(t)\mapsto C\gamma(t)C^{\operatorname{H}}$, for
$\forall C\in GL(n,\mathbb{C})$. As
$\frac{\operatorname{d}}{\operatorname{d}\!t}\gamma^{-1}(t)=-\gamma^{-1}(t)\dot{\gamma}(t)\gamma^{-1}(t)$,
one can readily see that this length is also invariant under
inversion.

Let the geodesic curve $\gamma(t)$ be
\begin{equation}\label{geod}
\gamma(t)=A^{\frac{1}{2}}\left(A^{-\frac{1}{2}}BA^{-\frac{1}{2}}\right)^tA^{\frac{1}{2}}\in
P(n)
\end{equation}
with $\gamma(0)=A,\gamma(1 )=B$ and
$\gamma'(0)=\ln\left(A^{-\frac{1}{2}}BA^{-\frac{1}{2}}\right)\in H(n)$.  Then
the midpoint of $A$ and $B$, denoted as $A\circ
B$, is given by
\begin{equation}\label{midpoint}
A\circ
B=A^{\frac{1}{2}}\left(A^{-\frac{1}{2}}BA^{-\frac{1}{2}}\right)^{\frac{1}{2}}A^{\frac{1}{2}}
\end{equation}
and the geodesic distance $d(A,B)$ can be computed explicitly by
\begin{equation}\label{riedis}
 d(A,B)=\|\ln(A^{-\frac{1}{2}}BA^{-\frac{1}{2}})\|_F=\left(\sum_{i=1}^n\ln(\lambda_i)^2\right)^\frac{1}{2},
\end{equation}
where  $\lambda_i$ are eigenvalues of
$A^{-\frac{1}{2}}BA^{-\frac{1}{2}}$. Since  $\lambda_i$ are also
eigenvalues of $A^{-1}B$, one can compute the distance $d(A,B)$, in
practice, without invoking the matrix square root
$A^{-\frac{1}{2}}$.

 \section{Control for positive definite Hermitian matrix systems}
\setcounter{figure}{0}

One of the purposes in control theory is to design the control
  input so that the output  approximates the target.
 Many algorithms for  specific approximation problems have been proposed (\cite{Yang,Cicco}). Among them, Zhang et al. \cite{zhang} proposed a steepest descent algorithm based on the natural gradient to design the controller of an open-loop stochastic distribution control system of multi-input and single output with a stochastic noise. In biomedicine field,  Hughes et al. \cite{Hughes}
 developed  a control law that can anticipate meals given a probabilistic description of the
 patient's eating behavior in the form of a random meal profile.

   In this section, the geodesic distance on manifold $P(n)$ is taken as the cost function to solve
    control problems of positive definite Hermitian matrix systems. We assume that the
  output matrix in $P(n)$ is only determined by the control input  through a system, allowing us to define the output as a matrix $A(u)\in P(n)$ as a function of the  input $u=(u^1, u^2,\ldots ,u^m)$.
 Our purpose
 is to design the control input $u$, such that $A(u)$ is as close as possible to another
 target positive definite Hermitian matrix $B$ (see Fig. \ref{fig:control}).

\begin{figure}[H]
\centering
\includegraphics[scale=0.55]{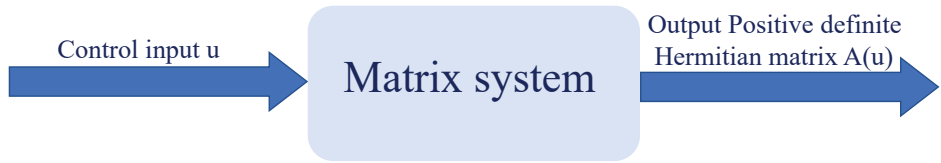}
\caption{Positive definite Hermitian matrix system}
\label{fig:control}
\end{figure}


The key points for designing an algorithm are then as followings:\\
1. To define a distance function to measure the difference between the  system output and  the target.\\
2. To computer the trajectory of input $u$ so that the  output system approximates the target, as close as possible.

In order to make the matrix $A(u)$  to be as close as possible to the
 given target matrix $B$, we use the geodesic distance (\ref{riedis})
 to measure the difference between the matrices $A(u)$ and $B$. Then  we are going to design a controller and obtain the $u_*$ such that
 \begin{equation}
 u_*=\arg\min_{u}J(u),
\end{equation}
where the cost function $J(u)$ is defined by
\begin{equation}\label{object}
 J(u)=d^2(A(u),B).
 \end{equation}

Let the system be well defined such that
\begin{equation}
\mathcal{M}=\left\{A(u)\mid u=(u^1, u^2,\ldots.,u^m)\in \Theta\subset \mathbb{R}^m\right\}
\end{equation}
is a submanifold   of manifold $P(n)$ where the control input $u$ plays
the role of local coordinates.

In the following,  the Riemannian gradient descent algorithm and  the natural gradient descent algorithm
 will be used to solve this control problem, respectively. Moreover, we will analyze the suitability of two algorithms
 in the illustrative examples. In fact, when the target  matrix $B$ lies on $\mathcal{M}$, both gradient descent algorithms
 are  applicable. One difference is that  the  Riemannian gradient descent algorithm realizes the optimisation of the
  input trajectory while the natural gradient algorithm converges faster than the former. When the target  matrix
  $B$ does not lie on $\mathcal{M}$, however, only the natural gradient algorithm is applicable.

\subsection{The Riemannian gradient descent algorithm}

Now we consider how to solve the control problem proposed
above using  the Riemannian gradient descent algorithm, in the case that the
target  matrix $B$ belongs to the output submanifold $\mathcal{M}$.

   Since both  the output $A(u)$ and the  target  matrix $B$ lie on  submanifold  $\mathcal{M}$, we
make use of the geodesic equation to derive the trajectory and the negative gradient of the cost function $J(u)$ about
  $A(u)$ as the  direction to give the  iterative formula.

\begin{thm}
For the control input $u=(u^1, u^2,\ldots ,u^m)$ of a given positive definite Hermitian
 matrix system, the iterative formula is given by
 \begin{equation}\label{alg1}
A(u_{k+1})=A^{\frac{1}{2}}(u_k)\exp\left\{-\eta_k\ln\left(A^{-\frac{1}{2}}(u_k)BA^{-\frac{1}{2}}(u_k)\right)\right\}A^{\frac{1}{2}}(u_k)
\end{equation}
where $\eta_k$ is the learning rate at time $k$.
\end{thm}

 \begin{proof} If the gradient of $J(u)$ about $A(u)$ is denoted by $\left(\nabla_AJ\right)(u)$, then (cf. \cite{Barba2})
  \begin{equation}\label{grad}
\left(\nabla_{A}J\right)(u_k)=A^{\frac{1}{2}}(u_k)\ln\left(A^{-\frac{1}{2}}(u_k)BA^{-\frac{1}{2}}(u_k)\right)A^{\frac{1}{2}}(u_k).
\end{equation}
Recall that the Riemannian exponential map $\exp_A$ on manifold $P(n)$ is defined by
 \begin{equation}
 \exp_A\{X\} = A^{\frac{1}{2}}\exp\left\{A^{-\frac{1}{2}}XA^{-\frac{1}{2}}\right\}A^{\frac{1}{2}},
 \end{equation}
 where $X$ is in the tangent space $T_AP(n)$. Then we obtain the iterative formula as
  \begin{equation}
  \begin{aligned}
A(u_{k+1})&=\exp_{A(u_k)}\left\{-\eta_k\left(\nabla_{A}J\right)(u_k)\right\}\\
&=A^{\frac{1}{2}}(u_k)\exp\left\{-\eta_k
A^{-\frac{1}{2}}(u_k)\left(\nabla_{A}J\right)(u_k)A^{-\frac{1}{2}}(u_k)\right\}A^{\frac{1}{2}}(u_k)\\
&=A^{\frac{1}{2}}(u_k)\exp\left\{-\eta_k\ln\left(A^{-\frac{1}{2}}(u_k)BA^{-\frac{1}{2}}(u_k)\right)\right\}A^{\frac{1}{2}}(u_k).
\end{aligned}
\end{equation}
This finishes the proof.
\end{proof}

Now, we give the Riemannian gradient descent algorithm for the proposed control problem for  positive definite Hermitian matrix systems.
\begin{alg}\label{alg11}
For the control  input $u=(u^1, u^2,\ldots ,u^m)$ on a given positive definite Hermitian
 matrix system, the iteration algorithm is  \\
1. Set $u_0=(u_0^1, u_0^2,\ldots ,u_0^m)$ as an initial input. Choose a fixed learning rate $\eta$ for simplicity and a desired tolerance $\varepsilon>0$.\\
2. At time $k$, calculate $A(u_k)$ using \eqref{alg1} and  $d({A(u_k)},B)$.\\
3. If $d({A(u_k)},B)<\varepsilon$ then stop. Otherwise, move to step $4$.\\
4. Increase $k$ by one and go back to step $2$.
 \end{alg}

\begin{rem}
The initial output matrix $A(u_0)$ will converge to the final output
 matrix $A(u_*)$ along the geodesic connecting them, hence this algorithm realizes the trajectory optimisation of input $u$.
 \end{rem}
 \begin{rem}
 If the target  matrix $B$ is not on submanifold  $\mathcal{M}$, it is impossible
 to find a geodesic on   submanifold  $\mathcal{M}$ such that it connects  the  output $A(u)$ and  the target $B$.
 Thus, at this time,  the Riemannian gradient algorithm needs be improved.
 \end{rem}

\subsection{Natural gradient descent algorithm}

The ordinary gradient, commonly used in learning methods on Euclidean spaces, does not give the steepest
direction of a cost function on manifold, but the natural
 gradient does. Next, we will first introduce an important lemma about
the natural gradient and then propose the natural gradient algorithm for the control problem.

Let $L(\theta)$ be a function defined in a Riemannian manifold parametrized by  $\theta\in\mathbb{R}^m$.
\begin{lem}[\cite{amari12}]  \label{ng}
  The natural gradient algorithm on a Riemannian manifold is given  by
\begin{equation}
\theta_{k+1}=\theta_k-\eta_kG^{-1}{\nabla}L(\theta_k),
\end{equation}
 where $G^{-1}=(g^{ij})$ is the inverse
of the Riemannian metric $G=(g_{ij})$, $L(\theta)$ is the cost
function and
\begin{equation}
{\nabla}L(\theta)=\left(\frac{\partial}{\partial\theta^1}L(\theta),
\frac{\partial}{\partial\theta^2}L(\theta),\ldots,\frac{\partial}{\partial\theta^m}L(\theta)\right).
\end{equation}
\end{lem}

In this subsection, we will give the natural gradient descent
algorithm for the considered system from the viewpoint of
information geometry. This algorithm can be applied no matter
whether  the target matrix $B$ is on the output submanifold  $\mathcal{M}$. The following lemma is useful for computing the gradient of cost function.


\begin{lem}[\cite{zhang2}]\label{lem}
Let $X(t)$ be a function-valued matrix of the real variable $t$ and let
$A,B$ be constant matrices. We assume that, for all $t$ in its
domain, $X(t)$ is an invertible matrix which does not have
eigenvalues on the closed negative real line. Then
\begin{equation}\label{lem1}
\frac{\operatorname{d}}{\operatorname{d}\!t}\operatorname{tr}\left(X^{\operatorname{T}}(t)X(t)\right)=2\operatorname{tr}\left(X^{\operatorname{T}}(t)\frac{\operatorname{d}}{\operatorname{d}\!t}X(t)\right),
\end{equation}
\begin{equation}\label{lem2}
\frac{\operatorname{d}}{\operatorname{d}\!t}\operatorname{tr}\left(\ln
X(t)\right)=\operatorname{tr}\left(X^{-1}(t)\frac{\operatorname{d}}{\operatorname{d}\!t}X(t)\right),
\end{equation}
\begin{equation}\label{lem3}
\frac{\operatorname{d}}{\operatorname{d}\!t}\operatorname{tr}\left(A
X(t)B\right)=\operatorname{tr}\left(A\frac{\operatorname{d}}{\operatorname{d}\!t}X(t)B\right).
\end{equation}
\end{lem}

Let $u=(u^1, u^2,\ldots,u^m)$ be a parameter space on which a cost function $J(u)$ is defined,
 we get the following theorem.
\begin{thm}\label{th1}
The iterative process on manifold $P(n)$ is given by
\begin{equation}\label{grad222}
u_{k+1}=u_k-\eta_k G^{-1}{\nabla}J(u_k),
\end{equation}
where the component of gradient
  ${\nabla}J(u_t)$   satisfies
\begin{equation}
\frac{\partial}{\partial u^i_k}J(u_k)
=2\operatorname{tr}\left(B^{-\frac{1}{2}}\ln \left(B^{-\frac{1}{2}}A(u_k)B^{-\frac{1}{2}}\right)B^{\frac{1}{2}}A^{-1}(u_k)\frac{\partial}{\partial
u_k^i}A(u_k)\right), ~~i=1,2,\ldots,m.
\end{equation}
\end{thm}

\begin{proof}
 According to Lemma \ref{ng}, we can get the iterative process as
\begin{equation}
u_{k+1}=u_k-\eta_k G^{-1}{\nabla}J(u_k),
\end{equation}
where the Fisher metric matrix $G$ is obtained by (\ref{metric}). Let
$X(u_k)=\ln\left(A^{-\frac{1}{2}}(u_k)BA^{-\frac{1}{2}}(u_k)\right)$.
 It is easy to show that $X(t)$ is  symmetric. Using Lemma \ref{lem}, we have the fact that
\begin{equation}
\begin{aligned}
\frac{\partial}{\partial u^i_k}J(u_k)
&=2\operatorname{tr}\left(\ln\left(B^{-\frac{1}{2}}A(u_k)B^{-\frac{1}{2}}\right)\frac{\partial}{\partial u_k^i}\ln\left(B^{-\frac{1}{2}}A(u_k)B^{-\frac{1}{2}}\right)\right)\\
&=2\operatorname{tr}\left(B^{-\frac{1}{2}}\ln\left(B^{-\frac{1}{2}}A(u_k)B^{-\frac{1}{2}}\right)B^{\frac{1}{2}}A^{-1}(u_k)\frac{\partial}{\partial u_k^i}A(u_k)\right),
\quad i=1,2,\ldots,m.
\end{aligned}
\end{equation}
This completes the proof of Theorem \ref{th1}.
\end{proof}

 From the above discussion, we formulate the natural gradient algorithm for the considered
 system as follows:
\begin{alg}\label{alg12}
 For the control of the input $u=(u^1, u^2,\ldots ,u^m)$ on the considered Hermitian positive definite
 matrix system, we have the steps that\\
 1. Set $u_0=(u_0^1, u_0^2,\ldots ,u_0^m)$ as an initial input. Choose a fixed learning rate $\eta$ and a desired tolerance $\varepsilon>0$.\\
 2. At time $k$, calculate $u_k$ using  \eqref{grad222} and ${\nabla}J(u_k)$.\\
 3. If $\|{\nabla}J(u_k)\|_F<\varepsilon$, stop. Otherwise, move to step $4$.\\
 4. Increase $k$ by one and go back to step $2$.

\end{alg}

\subsection{Simulations}

 From the following examples, we will show the efficiency of the two proposed algorithms, where the tolerance is $\varepsilon=10^{-15}$. In the first example, the target matrix $B$ lies in the submanifold $\mathcal{M}$, while it does not in the second example.

\textbf{Example 3.1.} We assume the  target matrix $B$ is a point of the output submanifold, so both
algorithms can be used. We choose a 3-dimensional input $u=(u^1, u^2, u^3)$ and define the matrix system as
\begin{equation}A(u)=
\begin{pmatrix}
   u^1   &iu^3\\
   -iu^3   &u^2
\end{pmatrix}, \quad u^1>0,u^1 u^2-(u^3)^2>0.
\end{equation}
The output submanifold  is then
\begin{equation}\mathcal{M}_1=\left\{A(y;u)=
\begin{pmatrix}
   u^1  &iu^3\\
  -iu^3 &u^2
\end{pmatrix}~\Big|~ u^1>0,u^1 u^2-(u^3)^2>0\right\}.
\end{equation}

We take $u_0^1=1, u_0^2=2,  u_0^3=1$ as the initial input $u_0$ and give  the target matrix $B$ by
\begin{equation}B=
\begin{pmatrix}
   55&2i\\
   -2i  &45
\end{pmatrix},
\end{equation}
so the coordinates of  the target point $B$ is $(55,45,2)$. Using Algorithm \ref{alg11} and  Algorithm  \ref{alg12}, the trajectories
  of the input $u$ from the initial state $A(u_0)$  to the target matrix $B$  are obtained  efficiently (see Fig. \ref{fig:32}).

\begin{figure}[H]
\centering
\begin{minipage}[h]{0.48\textwidth}
\centering
\includegraphics[width=8.9cm]{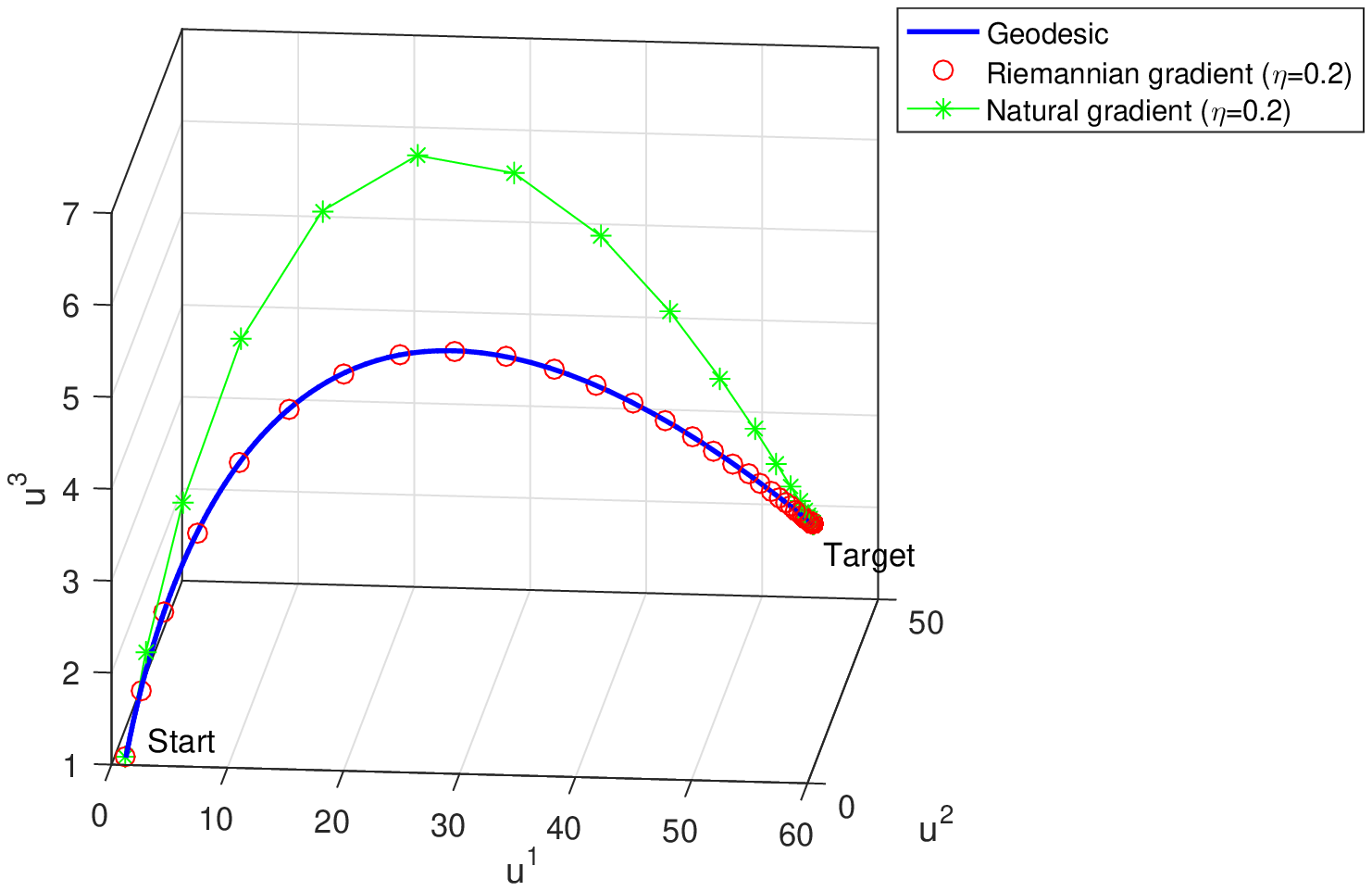}
\caption{Trajectory of  $u_k$}
\label{fig:32}
\end{minipage}\hspace{0.4cm}
\begin{minipage}[h]{0.48\textwidth}
\centering
\includegraphics[width=7.5cm]{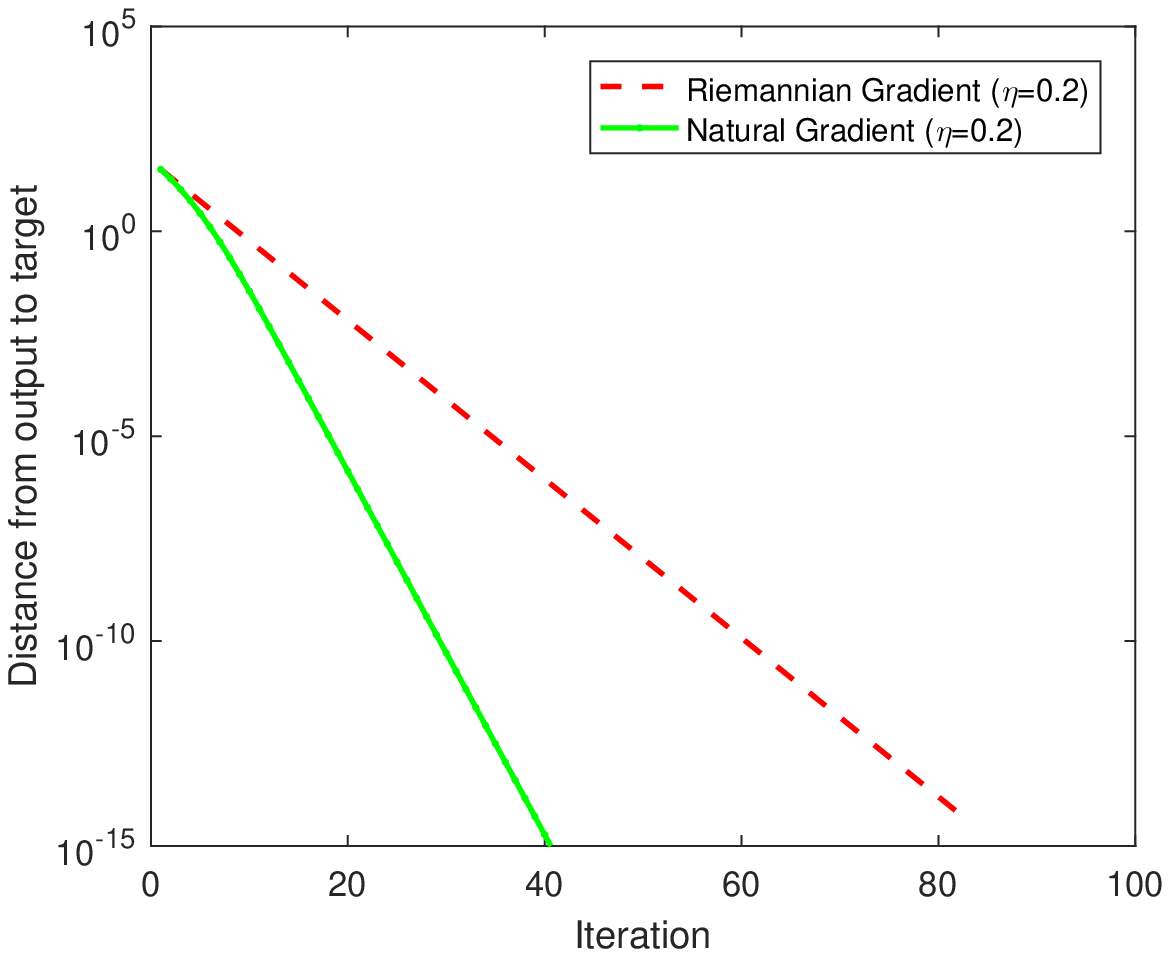}
\caption{The cost function $J(u_k)$ }
\label{fig:33}
\end{minipage}
\end{figure}


It is easy to see that the  trajectory of the input $u$  given by  the Riemannian gradient
algorithm is along the geodesic connecting the initial value $u_0$ and the target $B$ so that the path is the
shortest one. In addition, although the  trajectory of the input $u$  given by the natural gradient
algorithm is not optimal, the convergence is faster than the Riemannian gradient algorithm (see Fig. \ref{fig:33}).

\textbf{Example 3.2.} Now, we consider when the target matrix $B$
does not belong to the output submanifold. In this case,
only the natural gradient Algorithm \ref{alg12} is applicable to solve
the control problem. Setting the input $u$ to be a 2-dimensional
vector $(u^1, u^2)$ and the output matrix to be that
\begin{equation}A(u)=
\begin{pmatrix}
   u^1   &0\\
   0   &u^2
\end{pmatrix},\quad u^1,u^2>0,
\end{equation}
so  output submanifold is
\begin{equation}\mathcal{M}_2=\left\{A(u)=
\begin{pmatrix}
   u^1  &0\\
 0 &u^2
\end{pmatrix} ~\Big|~u^1,u^2>0\right\}.
\end{equation}

Let us take $u_0^1=1, u_0^2=4$ as the coordinates of the initial state $A(u_0)$  and  give the target matrix $B$ by
\begin{equation}B=
\begin{pmatrix}
   50&20i\\
   -20i  &40
\end{pmatrix}.
\end{equation}
Then,  using Algorithm \ref{alg12} to simulate the control process,
we obtain the trajectory of the input $u$ from the initial state
$A(u_0)$  to the approximate matrix $A(u_*)$ of  the target matrix
$B$ efficiently. The coordinate of $A(u_* )$ is
$u_*=(44.721,35.777)$ which can be taken as the geodesic
projection of the target $B$ onto submanifold  $\mathcal{M}_2$ (see
Fig. \ref{fig:34}). Furthermore, when we set the learning rate $\eta=0.1, 0.2,
0.5$ respectively, the efficiency and the convergence of Algorithm
\ref{alg12} are shown by Fig. \ref{fig:35}.

\begin{figure}[H]
\centering
\begin{minipage}[h]{0.48\textwidth}
\centering
\includegraphics[width=7.6cm]{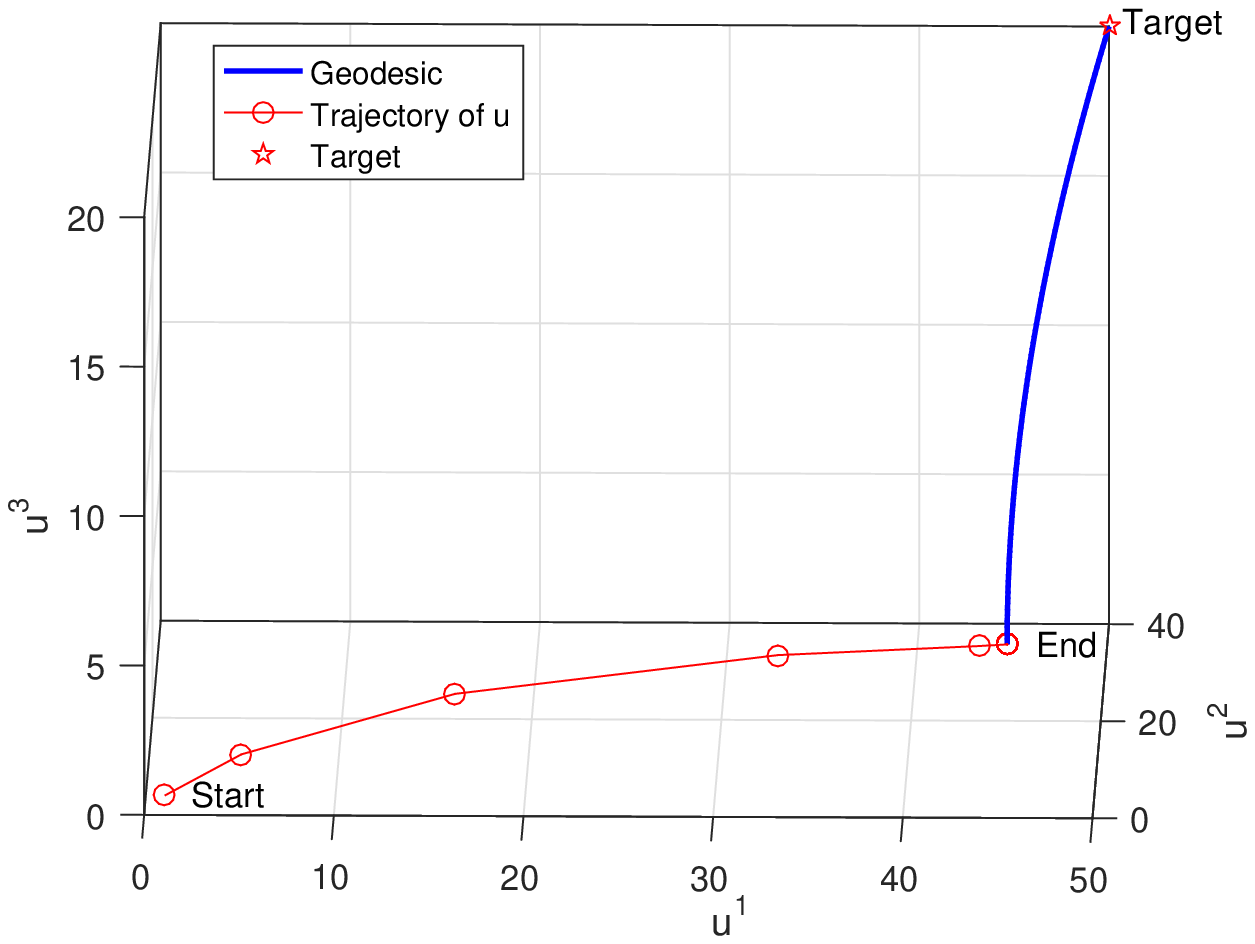}
\caption{Geodesic projection from target onto $\mathcal{M}_2$}
\label{fig:34}
\end{minipage}\hspace{0.4cm}
\begin{minipage}[h]{0.48\textwidth}
\centering
\includegraphics[width=7cm]{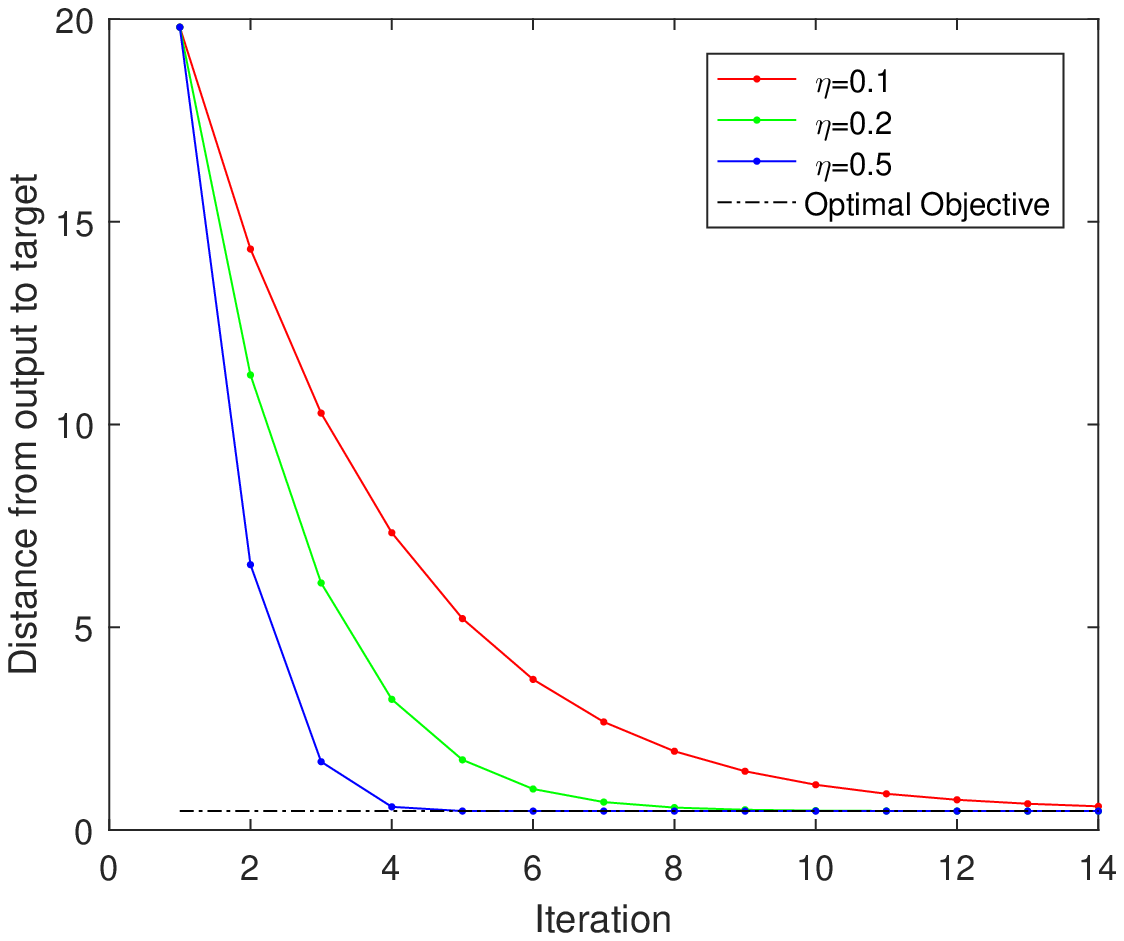}
\caption{Convergence of the natural gradient}
\label{fig:35}
\end{minipage}
\end{figure}


\section{Karcher mean of Toeplitz positive definite Hermitian matrices}

\setcounter{figure}{0}


Mean of matrices plays an important role in many fields, such as numerical analysis,
probability and statistics, engineering, biological and  social
sciences (cf. \cite{adam,Bhatia,liu,ryui}). Many algorithms have been developed to computer such means, see e.g. (\cite {moak2,Guven,Nobari}).  For electromagnetic or acoustic sensors, and more especially for radar,
lidar or echography, it is necessary to consider the spatial complex
data for the array processing or the time complex data for the
Doppler processing: $Z_n=[z_1, z_2,\ldots,z_n]^{\operatorname{T}}$(cf. {\cite{Barba2}}).
The covariance matrices of these complex data $R_n=E[Z_nZ_n^{\operatorname{H}}]$ are
the Toeplitz positive definite Hermitian matrices
\begin{equation}R_n=
\begin{pmatrix}
   r_0&\overline{r}_1&\cdots&\overline{r}_{n-1}\\
     r_1&r_0&\ddots&\vdots\\
   \vdots &\ddots&\ddots&\overline{r}_{1}\\
   r_{n-1}&\cdots&r_1&\overline{r}_{0}
\end{pmatrix},
\end{equation}
where $r_k=E[z_n\overline{z}_{n-k}]$ and $Z^{\operatorname{H}}R_nZ>0,$ for $\forall
Z\in \mathbb{C}^n$.

In this section, we consider the  complex circular multivariate
Gaussian distribution of zero mean with probability density function
\begin{equation}\label{gauss}
p(Z_n|R_n)=\frac{1}{\pi^n\det(R_n)}\exp\left\{-Z_n^{\operatorname{H}}R_n^{-1}Z_n\right\}.
\end{equation}
 Let us denote the set of all $n\times n$ Toeplitz
positive definite Hermitian matrices by $Sym(n,\mathbb{C})$ which is
obviously an $(n^2-n+1)$-dimensional submanifold of manifold $P(n)$.
In the Riemannian sense, the mean $\overline{R}$ of $N$ given
positive definite Hermitian matrices $R^1,R^2,\ldots,R^N$ is defined
as {\cite{karcher}}
\begin{equation}
\overline{R}=\arg\min_{R\in
Sym(n,\mathbb{C})}\frac{1}{N}\sum_{i=1}^Nd^2(R^i,R),
\end{equation}
which is called the Karcher mean. For $N$ distributions $p(\cdot|R^i), i=1,2,\ldots,N$,
 we denote their covariance matrices by $R^k$ and  define the cost
function by
\begin{equation}\label{obfu}
L(R)=\frac{1}{N}\sum_{i=1}^Nd^2(R^i,R).
\end{equation}
Note that
the local curvature of complex circular multivariate Gaussian distribution of zero
 mean is a non-positive constant, so the Karcher mean is unique \cite{karcher2}.

In \cite{Arnaudon}, it was shown that the Jacobi field for the
Karcher mean vanishes.  In order to estimate the Doppler
ambiance, Barbaresco {\cite{Barba2}} computed
the Jacobi field  and
 proposed  the Riemannian gradient algorithm to compute the Karcher Mean  by
\begin{equation}\label{grad1}
R_{k+1}=R_{k}^{\frac{1}{2}}\exp\left\{-\eta\sum_{i=1}^N\ln\left(R_{k}^{-\frac{1}{2}}R^iR_{k}^{-\frac{1}{2}}\right)\right\}R_{k}^{\frac{1}{2}}
\end{equation}
with $\eta$ the learning rate.

 In the following, we will propose the natural gradient algorithm to
 compute the Karcher mean of $N$ Toeplitz positive definite Hermitian matrices $R^i (i=1,2,\ldots,N)$ followed with simulations.

\subsection{Natural gradient descent algorithm}


 Let $\theta=(\theta^1, \theta^2,\ldots,\theta^m)$ be a parameter space on which a function $L(\theta)$
 is defined. Analogously to the proof of Theorem \ref{th1}, we have the following theorem:

\begin{thm}
The iterative process on manifold $Sym(n,\mathbb{C})$ is given by
\begin{equation}\label{grad2}
\theta_{k+1}=\theta_k-\eta G^{-1}{\nabla}|_{\theta=\theta_k}L(\theta),
\end{equation}
where
\begin{equation}
{\nabla}L(\theta)=\left(\frac{\partial}{\partial\theta^1}L(\theta),
\frac{\partial}{\partial\theta^2}L(\theta),\ldots,\frac{\partial}{\partial\theta^m}L(\theta)\right),
\end{equation}
and the component of gradient
  ${\nabla}L(\theta)$  satisfies
\begin{equation}
\frac{\partial}{\partial\theta^j}L(\theta)
=\frac{2}{N}\operatorname{tr}\left(\sum_{i=1}^NR_i^{-\frac{1}{2}}\ln\left(R_i^{-\frac{1}{2}}RR_i^{-\frac{1}{2}}\right)R_i^{\frac{1}{2}}R^{-1}\frac{\partial}{\partial\theta^j}R\right),\quad j=1,2,\ldots,m.
\end{equation}
\end{thm}

Now we are ready to formulate the natural gradient algorithm.
\begin{alg}\label{alg21}
The natural gradient algorithm to computer the Karcher mean of $N$ matrices of the manifold $Sym(n,\mathbb{C})$ is as follows:\\
 1. Take the  arithmetic mean  $\frac{1}{N}\sum_{i=1}^NR^i$ as the initial
 point $\theta_0$. Choose a learning rate $\eta$ and a desired tolerance $\varepsilon>0$.\\
 2. At time $k$, calculate $\theta_k$ using \eqref{grad2} and ${\nabla}L(\theta_k)$.\\
 3. If $\|{\nabla}L(\theta_k)\|<\varepsilon$, stop. Otherwise, move to step $4$.\\
 4. Increase $k$ by one and go back to step $2$.

\end{alg}

\subsection{Simulations}

In this subsection, we use the two gradient descent algorithms mentioned above to
compute Karcher mean of $N$ given Toeplitz positive definite
Hermitian matrices, where the tolerance is again $\varepsilon=10^{-15}$. From the following examples, it is shown
that the natural gradient algorithm is  more efficient than the
algorithm (\ref{grad1}).

For simplicity, we choose a 2-dimensional spatial complex data for the array processing or the
time complex data for the Doppler processing: $Z_2=[z_1, z_2]^{\operatorname{T}}$. Then,
 the  covariance matrices of these complex data can be written as
\begin{equation}
\begin{pmatrix}
   \theta_1&\theta_2+\operatorname{i}\theta_3\\
   \theta_2-\operatorname{i}\theta_3&\theta_1
\end{pmatrix}.
\end{equation}

\textbf{Example 4.1.}
We computer the Karcher mean of $R^1, R^2$ on manifold $Sym(2,\mathbb{C})$ as a first example, where
\begin{equation}
R^1=
\begin{pmatrix}
   5&1+2\operatorname{i}\\
  1-2\operatorname{i}&5
  \end{pmatrix},\quad
  R^2=
\begin{pmatrix}
  4&1-\operatorname{i}\\
  1+\operatorname{i}&4
\end{pmatrix}.
\end{equation}
In fact, it is easy to know their Karcher mean $R^1\circ R^2$ corresponds to (\ref{midpoint}).
 To show the algorithm's efficiency, the natural
gradient  algorithm and the algorithm
(\ref{grad1}) are used  respectively.  Choose the arithmetic mean
 $\frac{1}{2}\sum_{i=1}^2R^i$ as  the  initial
 point $\theta_0$. On manifold $Sym(2,\mathbb{C})$,  the
adjustment of the coordinate vector $\theta$ is given by
(\ref{grad1}) and (\ref{grad2}). From (\ref{geod}), the geodesic
between  $R^1$ and $R^2$ is computed (see Fig. \ref{fig:41}). We denote the Karcher mean by a red pentacle. It is shown that both algorithms converge to
the red pentacle. It is also shown that the natural gradient algorithm (\ref{grad2}) converges faster  than the algorithm
(\ref{grad1}) (see Fig. \ref{fig:42}).


\begin{figure}[H]
\centering
\begin{minipage}[h]{0.48\textwidth}
\centering
\includegraphics[width=11cm]{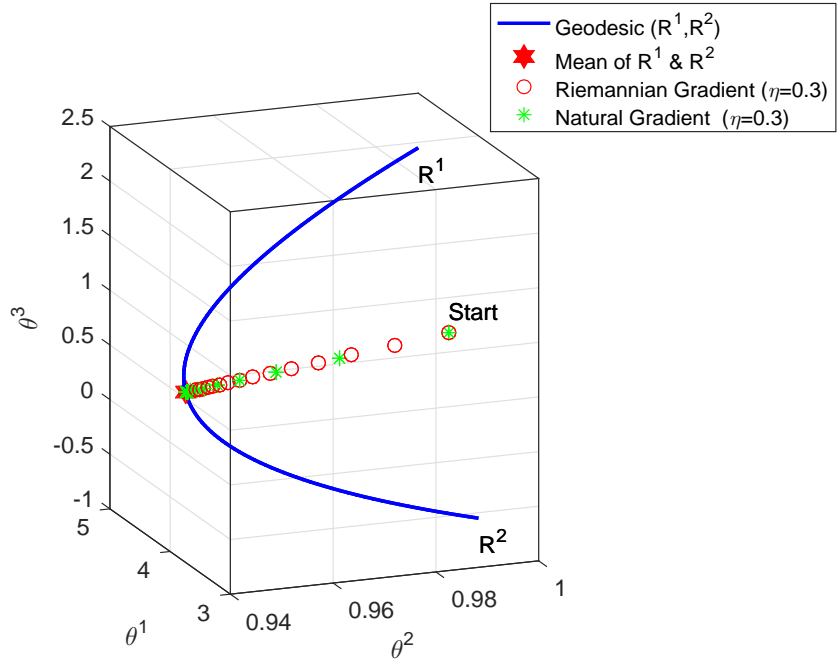}
\caption{Iterative process}
\label{fig:41}
\end{minipage}
\end{figure}

\begin{figure}[H]
\centering
\begin{minipage}[h]{0.48\textwidth}
\centering
\includegraphics[width=8cm]{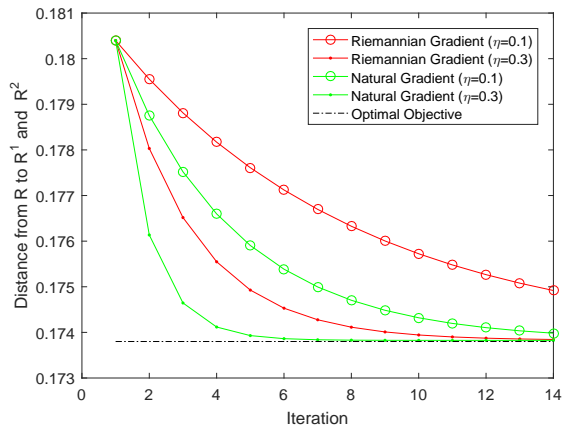}
\caption{Descent process  of the cost function}
\label{fig:42}
\end{minipage}
\end{figure}


\textbf{Example  4.2.} Here, we consider three $2\times2$ Toeplitz positive definite Hermitian
matrices
\begin{equation}
R^1=
\begin{pmatrix}
   3&1.5+2\operatorname{i}\\
  1.5-2\operatorname{i}&3
  \end{pmatrix},\quad
  R^2=
\begin{pmatrix}
  2&1-\operatorname{i}\\
  1+\operatorname{i}&2
\end{pmatrix},\quad
R^3=
\begin{pmatrix}
   4&1+2\operatorname{i}\\
  1-2\operatorname{i}&4
\end{pmatrix}.
\end{equation}
Using (\ref{geod}), we can get the geodesics  between each two of the three points $R^1,R^2$ and $R^3$ on
$Sym(2,\mathbb{C})$, which form a geodesic triangle (see Fig. \ref{fig:43}).  The midpoint
of each geodesic is  obtained using (\ref{midpoint}). Thus,  each
median connects a vertex with the midpoint of its opposing side.
In Euclidean spaces, these centerlines always meet in a single point which is the Karcher mean.
 However, in curved spaces, as Fig. \ref{fig:43}
shows, this is no longer true. In fact, when we change the look angle to Fig. \ref{fig:43}, it is shown that
three midlines are very likely non-coplanar (see Fig. \ref{fig:44}).

\begin{figure}[H]
\centering
\begin{minipage}[h]{0.48\textwidth}
\centering
\includegraphics[width=7.2cm]{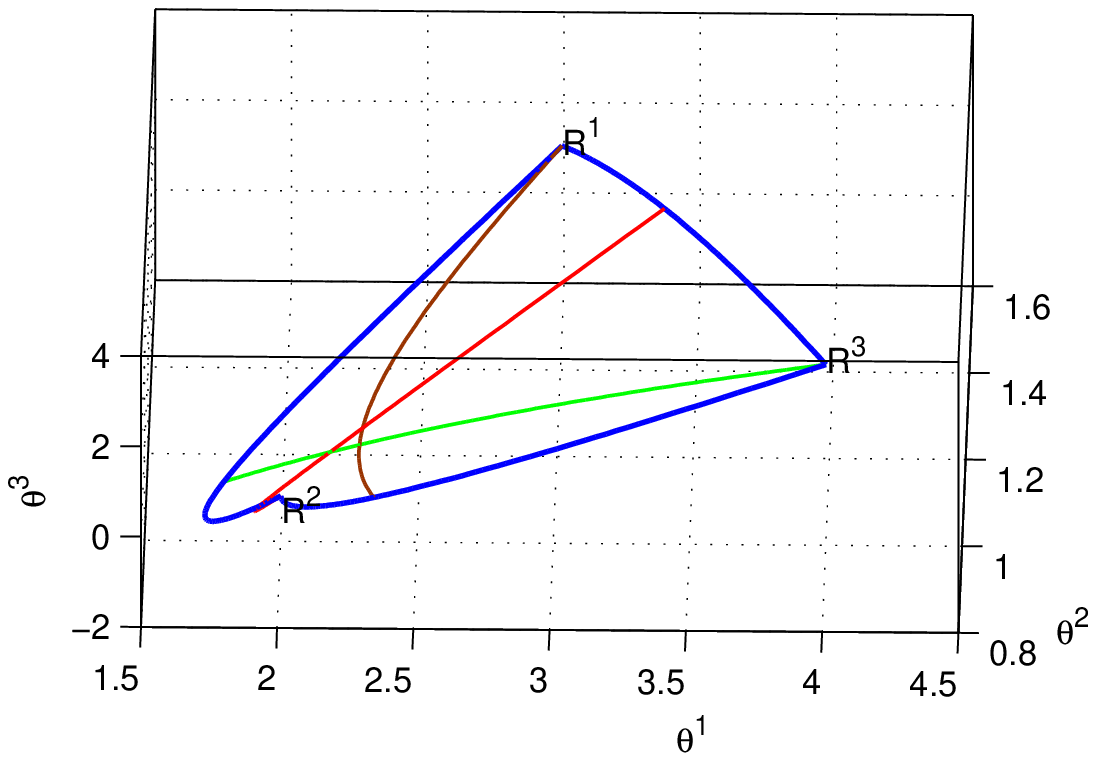}
\caption{Geodesic triangle (position 1)}
\label{fig:43}
\end{minipage}\hspace{0.4cm}
\begin{minipage}[h]{0.48\textwidth}
\centering
\includegraphics[width=7.8cm]{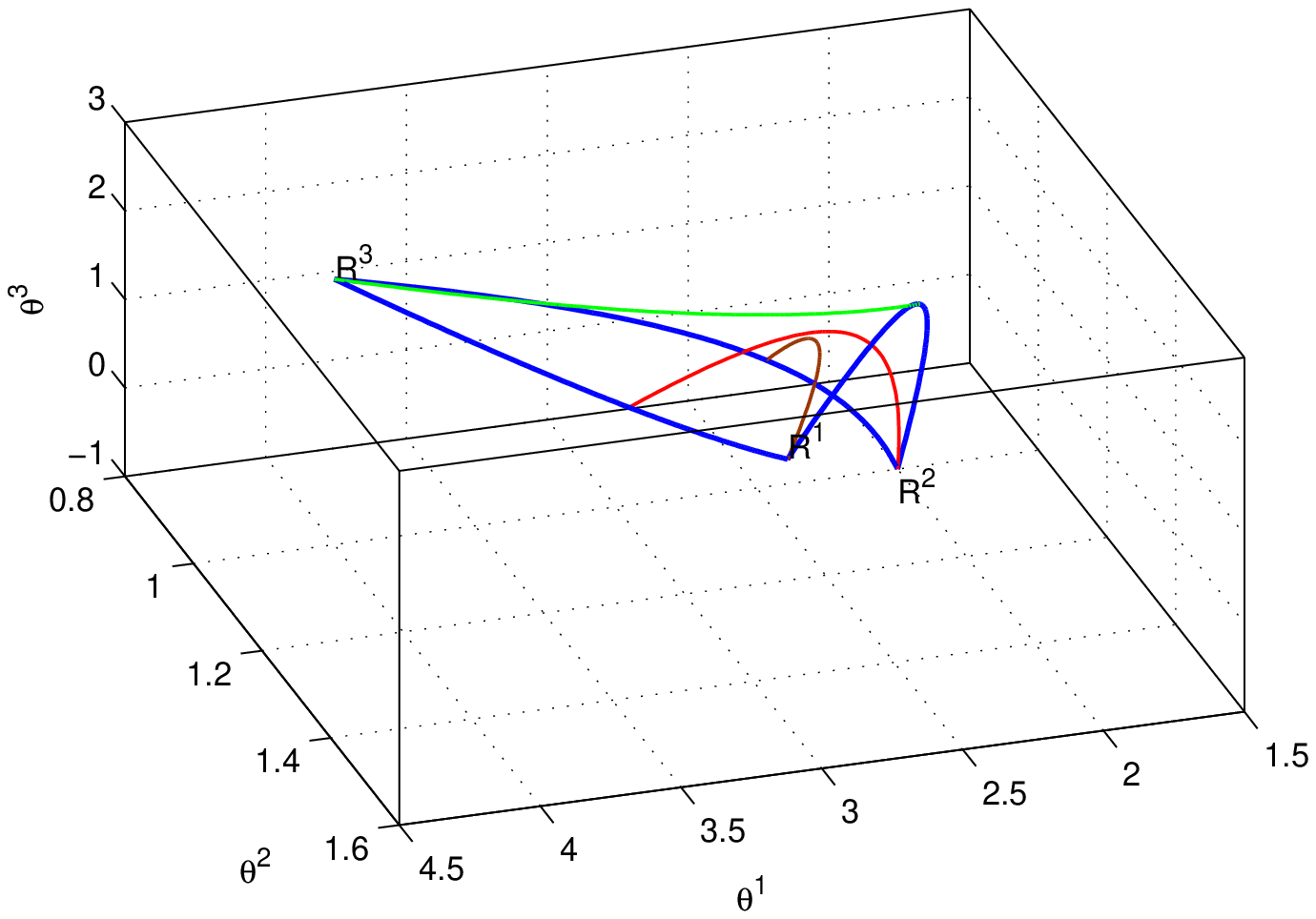}
\caption{Geodesic triangle (position 2)}
\label{fig:44}
\end{minipage}
\end{figure}


Now we   compute   the Karcher mean $p(\cdot|\overline{R})$ of 3
distributions $p(\cdot|R^i),i=1,2,3$ using these two algorithms. Similarly to Example 4.1, first we still take the
  arithmetic mean
 $\frac{1}{3}\sum_{i=1}^3R^i$ as the initial point $\theta_0$ .  On manifold $Sym(2,\mathbb{C})$,  the
adjustment of the coordinate vector $\theta$ is given by
(\ref{grad1}) and (\ref{grad2}). Finally,  both algorithms
converge to the Karcher mean  (see Fig. \ref{fig:45})
\begin{equation}
\overline{R}=\arg\min_{R\in
Sym(2,\mathbb{C})}\frac{1}{3}\sum_{i=1}^3d^2(R^i,R)=
\begin{pmatrix}
   2.295&0.980+0.617\operatorname{i}\\
  0.980-0.617\operatorname{i}&2.295
  \end{pmatrix}.
\end{equation}
Again the natural gradient  algorithm is   faster
 than the  algorithm (\ref{grad1})  (see Fig. \ref{fig:46}).


\begin{figure}[H]
\centering
\begin{minipage}[h]{0.48\textwidth}
\centering
\includegraphics[width=9.5cm]{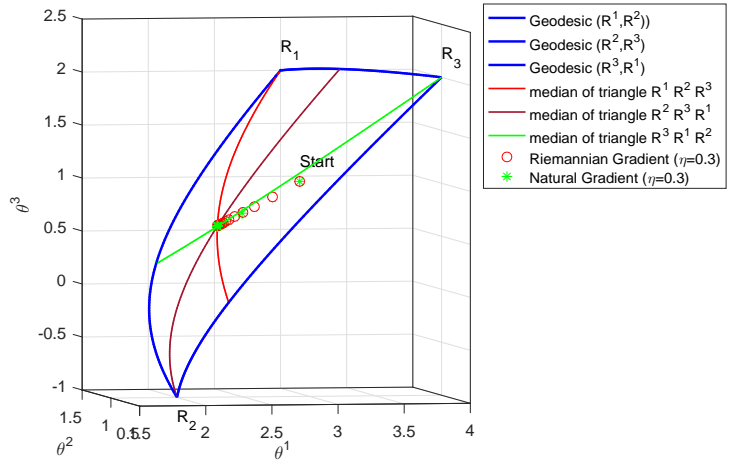}
\caption{Convergence from initial point to the Karcher mean}
\label{fig:45}
\end{minipage}\hspace{0.4cm}
\end{figure}

\begin{figure}[H]
\centering
\begin{minipage}[h]{0.48\textwidth}
\centering
\includegraphics[width=7.5cm]{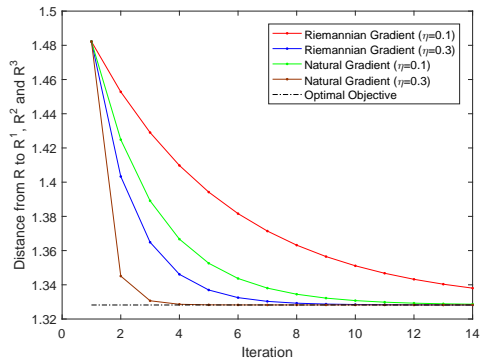}
\caption{Descent process  of the cost function}
\label{fig:46}
\end{minipage}
\end{figure}

\section{Conclusions}

 In this paper, we applied  the Riemannian gradient algorithm and the natural gradient algorithm to the control of positive Hermitian matrix systems as well as the computation of Karcher mean of Toeplitz positive definite Hermitian matrices. Their behaviors were also compared.

  For the control system, when the  target matrix belongs to the output submanifold,
   the  Riemannian gradient descent algorithm and the natural gradient algorithm are both applicable. It was shown that the
 former realizes the optimisation of   input trajectory while the latter has preferable convergence.
  We also proposed the natural gradient algorithm to compute
the  Karcher mean of matrices in the submanifold $Sym(n,\mathbb{C})$, taking    the sum of  geodesic distances as the cost function.
The simulations showed that convergence rate of the natural gradient algorithm is faster than the Riemannian gradient algorithm, which has been widely used to solve such optimisation problems during the last decade.

\section*{ Acknowledgements}
X.D. is supported by the National Natural Science Foundation of China (No. 61401058) and the Natural Science Foundation of Liaoning Province (No. 20180550112). H.S. is partially supported by the National Natural Science Foundation of China (No. 61179031, No. 10932002). L.P. is supported by JSPS Grant-in-Aid for Scientific Research (No. 16KT0024),  the MEXT  ``Top Global University Project", Waseda University Grant for Special Research Projects (No. 2019C-179, No. 2019E-036) and Waseda University Grant Program for Promotion of International Joint Research.




\end{document}